\documentclass[12pt,reqno]{amsart}
\usepackage{fullpage}
\usepackage{times}

\newtheorem{theorem}{Theorem} 
\newtheorem{lemma}[theorem]{Lemma}
\newtheorem{prop}[theorem]{Proposition}

\theoremstyle{definition}

\renewcommand{\mod}[1]{{\ifmmode\text{\rm\ (mod~$#1$)}\else\discretionary{}{}{\hbox{ }}\rm(mod~$#1$)\fi}}
\newcommand{\ep}{\varepsilon}
\newsymbol\nmid 232D

\newcommand{\Bval}{\lfloor \frac32\log_2 x \rfloor}

\newcommand{\A}{{\mathcal A}}

\newcommand{\N}{{\mathbb N}}

\renewcommand{\S}{{\mathcal S}}

\begin{document}

\title{Primitive sets with large counting functions}
\author[Greg Martin and Carl Pomerance]{Greg Martin and Carl Pomerance \\ Draft: \today}

\dedicatory
{Dedicated to Andr\'as S\'ark\"ozy on his 70th birthday}
\maketitle

\begin{abstract}
A set of positive integers is said to be primitive if no element of the set 
is a multiple of another. If $\S$ is a primitive set and $S(x)$ is 
the number of elements of $\S$ not exceeding $x$, then a result of 
Erd\H os implies that $\int_2^\infty( S(t)/t^2\log t) \,dt$ converges. 
We establish an approximate converse to this theorem, showing that if 
$F$ satisfies some mild conditions and $\int_2^\infty( F(t)/t^2\log t) \,dt$ 
converges, then there is a primitive set $\S$ with $S(x)\asymp F(x)$.
\end{abstract}


\section{Introduction}
\label{intro}

A set of positive integers is {\em primitive} if no element of 
the set is a multiple of another.   In the 1930s Chowla, Davenport, and
Erd\H os independently studied a special primitive set, namely
the set of primitive nondeficient numbers 
(numbers $n$ such that the sum of the proper divisors of $n$ is at least $n$,
but no proper divisor of $n$ has this property), which probably inspired
the generalization to general primitive sets around the same time.
Besicovitch~\cite{Be} showed, perhaps unexpectedly, that the upper asymptotic density
of a primitive set can be arbitrarily close to $1/2$; his construction yields a set whose counting function is occasionally large but usually extremely small. In~\cite{EP35}, Erd\H os showed that the lower asymptotic density of a primitive set must be~0, and also that
\begin{equation}
\label{eq-nlogn}
\sup_{\S~{\rm primitive}}\sum_{n\in\S\setminus\{1\}}\frac1{n\log n}<\infty.
\end{equation}
It is thought that this supremum is attained
when $\S$ is the set of primes, but this is still not known. 
Further references to results on primitive sets can be found 
in~\cite{HR},~\cite[Section~5.1]{PS}, and~\cite[Section~5]{Ru}. 

In this note we ask if there are primitive sets with consistently large 
counting functions (as opposed to occasionally large counting functions, as in 
Besicovitch's example).  We show that essentially any smoothly growing
counting function that is consistent with the necessary convergence~\eqref{eq-nlogn} 
can be the order of magnitude for the counting function of a primitive set.

A favorite problem of Erd\H os, as related in \cite{ESS}, is as follows:
If $1<b_1<b_2<\dots$ is a sequence of numbers with $\sum1/b_n\log b_n<\infty$,
must there exist a primitive sequence $1<a_1<a_2<\dots$ with $a_n\ll b_n$?
One may interpret our principal result as answering ``yes" for smoothly growing
sequences $\{b_n\}$.

For a set $\S$ of natural numbers, let $S(x)$ denote its counting function;
that is, $S(x)$ is the number of members of $\S$ not exceending $x$.
Let $\log_1 x = \max\{1,\log x\}$ and $\log_\ell x = \log_1(\log_{\ell-1} x)$ 
for every integer $\ell\ge2$.

\begin{theorem}
\label{main}
Suppose that $L(x)$ is defined, positive, and increasing for $x\ge2$,
that $L(2x)\sim L(x)$ as $x\to\infty$, and that
\begin{equation}
\label{eq-Lint}
\int_2^\infty\frac{dt}{t\log t\cdot L(t)}<\infty.
\end{equation}
Then there is a primitive set $\S$ such that
\begin{equation}
\label{eq-Sbound}
S(x) \asymp \frac x{\log_2 x \cdot \log_3 x \cdot L(\log_2 x)}
\end{equation}
for all sufficiently large~$x$.  In particular, for any integer $\ell\ge3$
and every real number $\ep>0$, there exists a primitive set $\S$ such that
\begin{equation}
\label{eq-logs}
S(x) \asymp \frac{x}{\log_2 x \cdots \log_{\ell-1} x \cdot (\log_\ell x)^{1+\ep}}
\end{equation}
for all sufficiently large $x$.
\end{theorem}

By taking $L(x)=(\log_2x)\cdots(\log_{\ell-3}x) (\log_{\ell-2}x)^{1+\ep}$,
we see that \eqref{eq-Sbound} implies \eqref{eq-logs} for $\ell\ge4$,
and the case $\ell=3$ follows by taking $L(x)=(\log x)^\ep$.
By an argument somewhat similar to our proof of Theorem~\ref{main},
Ahlswede, Khachatrian, and S\'ark\"ozy~\cite{AKS} gave a construction
for the lower bound in \eqref{eq-logs}
in the case $\ell=3$. Like the paper \cite{AKS}, our proof depends heavily on a result
of Sathe--Selberg on the fine distribution of integers with a given
number of prime factors.

It is not hard to see that the condition \eqref{eq-Lint} is necessary
in Theorem~\ref{main}.  Indeed, suppose $\S$ is a set of natural numbers
greater than 1 satisfying \eqref{eq-Sbound}, and suppose that $\sum_{n\in\S\setminus\{1\}}1/(n\log n)$
converges (as it must, by equation~\eqref{eq-nlogn}, for primitive sets $\S$).  Since
$$
\sum_{n\in\S\setminus\{1\}}\frac1{n\log n}=\int_2^\infty S(t)\left(\frac1{t^2\log t}
+\frac1{t^2\log^2t}\right)\,dt,
$$
it follows that
$$
\int_2^\infty \frac{S(t)}{t^2\log t}\,dt<\infty.
$$
Then \eqref{eq-Sbound} implies that
$$
\int_2^\infty \frac{dt}{t\log t\cdot\log_2t\cdot\log_3t\cdot L(\log_2t)}\,dt<\infty.
$$
Via a change of variables, we obtain \eqref{eq-Lint}.

Another question one might consider is what conditions on the distribution
of a set $\A$ of natural numbers forces $\A$ to have a large primitive subset.
It is not too difficult to see that if an infinite set $\A$ contains no 
primitive subset of size $k$, then $A(x)\ll k\log x$.  Indeed, if
$b_1<\dots<b_k$ are any $k$ consecutive elements in $\A$, that they
are not primitive forces some $b_i\mid b_j$ for $1\le i<j\le k$, so
that $b_k/b_1\ge2$.  On the other hand, the set $\A=\{m2^j:m<2k-1,j\ge0\}$
has no primitive subset of size $k$ and $A(x)\gg k\log x$.

At the other extreme, it is also
not difficult to see that if $\A$ has positive upper density, then it contains
a primitive subset also with positive upper density.  Indeed, 
any integer subset of
a dyadic interval $[x,2x)$ is primitive, and a set with positive upper density
must contain a fixed positive proportion $\delta$ of each dyadic interval 
$[x_i,2x_i)$ for some unbounded sequence $\{x_i\}$.  
The Besicovitch argument then goes over to show that $\A$ contains a primitive
subset of upper density arbitrarily close to $\delta/2$.

We address this subset question for a set of ``intermediate" density, namely it has
density~0, but an infinite reciprocal sum.  We prove the following
result.
\begin{theorem}
\label{theorem2}
There is a set $\A$ of natural numbers of asymptotic density $0$ satisfying
\begin{equation}
\label{eq-sums}
\sum_{a\in\A\setminus\{1\}}\frac1{a\log a}<\infty 
\quad\text{and}\quad \sum_{a\in\A}\frac1a=\infty,
\end{equation}
such that for any primitive set $\S$ contained in $\A$ we have
\begin{equation}
\label{eq-primsum}
\sum_{s\in\S}\frac1s<\infty.
\end{equation}
\end{theorem}
In particular, no primitive subset of $\A$ has positive relative lower 
density in $\A$, despite the counting function of $\A$ being small enough to allow
the possibility.
The set $\A$ that we exhibit has the property that there is a primitive
subset of relative positive upper density,
so there remains a perhaps interesting problem:  Is there a set $\A$
with infinite reciprocal sum such that any primitive subset has relative
density 0 in $\A$?  Maybe the Besicovitch construction will show
such a set $\A$ does not exist.

\section{Constructing primitive sets from a sequence of primes}
\label{construction}

Let $p_1<p_2<\cdots$ be any infinite sequence of primes such that 
\[
 \sum_{j=1}^\infty \frac1{p_j} < \frac12.
\]
We need this sequence not to grow too quickly; for now we make only the restriction $p_j \ll j^2$.

Using the usual notation $\Omega(n)$ for the number of prime factors of $n$ counted with multiplicity, we define for any positive integer $k$
\[
 \S_k = \{ n\in\N\colon \Omega(n) = k,\, p_k\mid n,\, (p_1\cdots p_{k-1},n)=1 \},
\]
and we set
\[
 \S = \bigcup_ {k=1} ^\infty \S_k.
\]
We prove two results about $\S$: the first is that $\S$ is primitive and the
second is a lower bound for $S(x)$ (see Proposition~\ref{sum over k prop} below).

\begin{lemma}
\label{Sprim}
The set $\S$ is primitive.
\end{lemma}
\begin{proof}
Note that if $m$ and $n$ are distinct positive integers and $m$ divides $n$, 
then $\Omega(m) < \Omega(n)$. Therefore if $\S$ were not primitive, then 
there would exist positive integers $j<k$ and integers $m\in\S_j$ and 
$n\in\S_k$ such that $m\mid n$. However, then $p_j$ would divide $m$ but not $n$, 
a contradiction. (Indeed, $\S$ is an example of a homogeneous set, in the terminology of~\cite{Zh}.)
\end{proof}

Let $\sigma_j(x)$ denote the number of positive integers $n\le x$ such that $\Omega(n)=j$. 

\begin{lemma}[Sathe--Selberg]
For any positive integer $j\le\Bval$,
\[
\sigma_j(x) = H_j(x)\left(1+O\left(\frac1{\log_2x}\right)\right)
\]
where
\[
H_j(x)=G\bigg( \frac{j-1}{\log\log x} \bigg) \frac x{\log x} \frac{(\log\log x)^{j-1}}{(j-1)!} 
~\hbox{ and }~
G(z) = \frac1{\Gamma(z+1)} \prod_p \bigg( 1-\frac zp \bigg)^{-1} \bigg( 1-\frac1p \bigg)^z.
\]
\label{SS lemma}
\end{lemma}
For a proof, see \cite[Theorem 7.19]{MV}.

\begin{lemma}
Let $x$ be a sufficiently large real number.
For any integer $k\in[2,\frac32\log_2x]$,
\[
 S_k(x) \asymp \frac x{\log x} \frac{(\log\log x)^{k-2}}{(k-2)!} \frac1{p_k},
\]
where the implied constants are absolute.
\label{Skx lemma}
\end{lemma}

\begin{proof}
The result follows immediately from the prime number theorem in the case $k=2$,
so assume that $k\ge3$.
Since every element of $S_k(x)$ is divisible by $p_k$ and is coprime to
$p_1\dots p_{k-1}$, we have the inequalities
\[
\sigma_{k-1}\left(\frac{x}{p_k}\right)\ge
S_k(x) \ge \sigma_{k-1}\bigg( \frac x{p_k} \bigg) - \sum_{j=1}^{k-1} \sigma_{k-2}\bigg( \frac x{p_jp_k} \bigg).
\]
By Lemma~\ref{SS lemma}, this becomes
\begin{align*}
H_{k-1}\left(\frac{x}{p_k}\right)&\left(1+O\left(\frac1{\log_2(x/p_k)}\right)\right)
\ge S_k(x)\\
&\ge
\left(H_{k-1}\left(\frac{x}{p_k}\right)-\sum_{j=1}^{k-1}H_{k-2}\left(\frac{x}{p_jp_k}
\right)\right)\left(1+O\left(\frac1{\log_2(x/p_jp_k)}\right)\right).
\end{align*}
Because $k\ll\log_2 x$ and $p_j \ll j^2$, each occurrence of $\log(x/p_k)$
or $\log(x/p_jp_k)$ can be rewritten as $(\log x)(1 + O(1/\log_2 x))$, 
and similarly $\log_2(x/p_k)$ and $\log_2(x/p_jp_k)$
can be rewritten as $(\log_2 x)(1+O(1/\log x))$. 
In addition, the expressions $G((k-2)/\log_2(x/p_k))$ and
$G((k-3)/\log_2(x/p_jp_k))$ can be rewritten as 
$$
G\left(\frac{k-2}{\log_2 x} + O\left(\frac1{\log_2 x}\right)\right)=
G\left(\frac{k-2}{\log_2 x}\right)\left(1+O\left(\frac1{\log_2 x}\right)\right),
$$
since $\log G(z)$ is analytic 
and hence has a bounded first derivative in a neighborhood of the interval $[0,3/2]$.
Therefore
\begin{align*}
H_{k-1}\left(\frac{x}{p_k}\right)&\left(1+O\left(\frac1{\log_2x}\right)\right)
\ge S_k(x)\\
&\ge
H_{k-1}\left(\frac{x}{p_k}\right)
\left(1-\frac{k-3}{\log_2x}\sum_{j=1}^{k-1}\frac1{p_j}\right)
\left(1+O\left(\frac1{\log_2x}\right)\right).
\end{align*}
Since the sum is less than $\frac12$, and since $G(z)$ is bounded away from~0 
and~$\infty$ on the interval $[0,3/2]$, this becomes
\[
S_k(x) \asymp H_{k-1}\left(\frac{x}{p_k}\right)
\asymp \frac x{\log x} \frac{(\log\log x)^{k-2}}{(k-2)!} \frac1{p_k}
\]
as claimed.
\end{proof}

\begin{prop}
For $x\ge p_1$, we have $x/p_B\gg S(x) \gg x/p_{B'}$, where 
$B=B(x)=\lfloor\frac12\log_2x\rfloor$ and $B'=B'(x)= \Bval$.
\label{sum over k prop}
\end{prop}

\begin{proof}
Since $\S = \bigcup_{k=1}^\infty \S_k$ is a disjoint union, 
we have by Lemma~\ref{Skx lemma},
$$
S(x) \ge \sum_{k=2}^{B'} S_k(x) 
\gg \sum_{k=2}^{B'} \frac x{\log x} \frac{(\log_2 x)^{k-2}}{(k-2)!}\frac1{p_k}
\ge \frac x{\log x} \frac1{p_{B'}} \sum_{k=2}^{B'} \frac{(\log_2 x)^{k-2}}{(k-2)!}
\gg \frac{x}{p_{B'}},
$$
where we used the inequality
$$
\sum_{j=0}^{\lfloor y\rfloor}\frac{y^{j}}{j!}\gg e^y
$$
(which follows from \cite[equation 1.10]{No} with $\beta=0$) for the last step.
For the upper bound, we have
\[
S(x)\le\sum_{k=1}^\infty S_k(x)\le\sum_{k=B+1}^{B'}S_k(x)
+\sum_{\substack{n\le x\\\Omega(n)\le B}}1
+\sum_{\substack{n\le x\\\Omega(n)>B'}}1.
\]
There is a positive constant $c$ such that the last two sums here are 
$O(x/(\log x)^{c})$.  Indeed, $\Omega(n)\le B$ implies that $\omega(n)\le B$,
where $\omega$ counts the number of distinct prime divisors, so the estimate
for $\Omega(n)\le B$ follows from the Hardy--Ramanujan inequality
(see~\cite[Proposition 3]{EN}).
If $\Omega(n)>B'$, a similar estimate holds using the Hardy--Ramanujan
inequality plus an estimate for those $n$ with $\Omega(n)-\omega(n)$ large,
or more directly from~\cite[Lemma 13]{LP}.

By Lemma~\ref{Skx lemma},
\[
\sum_{k=B+1}^{B'}S_k(x)\ll\sum_{k=B+1}^{B'}\frac{x}{\log x}\frac{(\log_2x)^{k-2}}{(k-2)!}
\frac1{p_B}
\le\frac{x}{p_B}\sum_{j=0}^\infty\frac{(\log_2x)^j}{j!\log x}=\frac{x}{p_B}.
\]
Since $p_B\le p_{B'}=O(B'^2)=O((\log_2x)^2)$, sets of size $O(x/(\log x)^{c})$
are negligible, and our result follows.
\end{proof}

\section{Proof of Theorem \ref{main}}

\begin{lemma}
\label{bound on L lemma}
Suppose that $L(x)$ is defined, positive, and increasing for $x\ge2$ and that
$L(2x)\sim L(x)$ as $x\to\infty$. Then $L(x) \ll_\ep x^\ep$ for any $\ep>0$.
\end{lemma}

\begin{proof}
Given $\ep>0$, we need to show that $L(x)/x^\ep$ is bounded. Since $L(2x)\sim L(x)$, we may choose $x_1$ such that $L(2x) < (1+\ep\log2) L(x)$ for all $x\ge x_1$. Define $M_u = \max_{u\le x\le 2u} L(x)/x^\ep$. Then for any $u\ge x_1$,
\[
M_{2u} = \max_{2u\le x\le 4u} \frac{L(x)}{x^\ep} = \max_{u\le y\le 2u} \frac{L(2y)}{(2y)^\ep} < \frac{1+\ep\log 2}{2^\ep} \max_{u\le y\le 2u} \frac{L(y)}{y^\ep} < 1\cdot M_u,
\]
since $2^\ep > 1+\ep\log 2$. Therefore $M_{x_1} > M_{2x_1} > M_{4x_1} > \cdots$, 
and so $L(x)/x^\ep$ is bounded by $M_{x_1}$ on $[x_1,\infty)$. 
Since it is clearly bounded by $L(x_1)$ on $[2,x_1]$, the lemma is established.
\end{proof}

\begin{prop}
\label{exist prime sequence prop}
Suppose that $L(x)$ is defined, positive, and increasing for $x\ge2$, that
$L(2x)\sim L(x)$ as $x\to\infty$, and that
\begin{equation*}
\int_2^\infty\frac{dt}{t\log t\cdot L(t)}<\infty.
\end{equation*}
Then there is a sequence $p_1<p_2<\cdots$ of primes with $\sum_{k=1}^\infty 1/p_k<1/2$ and 
$p_k \sim k \log k \cdot L(k)$ as $k\to\infty$.
\end{prop}

\begin{proof}
Choosing $y_0$ so that $L(y)\ge1$ holds for all $y\ge y_0$, define
\[
q_k = \begin{cases}
\text{the $k$th prime}, &\text{if } k < y_0, \\
\text{the $\lfloor k L(k) \rfloor$th prime}, &\text{if } k\ge y_0.
\end{cases}
\]
Then $\{q_k\}$ is increasing since
$(k+1)L(k+1) \ge (k+1)L(k) \ge kL(k)+1$ for $k\ge y_0$, so that 
$\lfloor (k+1)L(k+1) \rfloor > \lfloor kL(k) \rfloor$. 
By the prime number theorem, when $k\to\infty$ we have
\[
q_k \sim \lfloor kL(k) \rfloor \log \lfloor kL(k) \rfloor \sim kL(k) (\log k + \log L(k)) \sim kL(k)\log k,
\]
where the last asymptotic equality used Lemma~\ref {bound on L lemma}. 
Further, 
\[
\sum_{k\ge y_0+1} \frac1{q_k}\ll\sum_{k\ge y_0+1}\frac1{k\log k\cdot L(k)} 
< \int_{y_0}^\infty \frac{dt}{t\log t\cdot L(t)}
\]
which converges; thus there is some nonnegative integer $k_0$ such that
$\sum_{k>k_0}1/q_k<1/2$. Then the sequence $\{p_k\}$ defined by  $p_k=q_{k_0+k}$ has the required properties.
\end{proof}

\noindent{\it Proof of Theorem \ref{main}}.  Note that if $c>0$ is fixed,
\[
p_{\lfloor c\log_2x\rfloor} \sim c\log_2 x \cdot \log_3 x 
\cdot L\big( c \log_2x \big) \sim c \log_2x \cdot \log_3x \cdot L(\log_2 x)
\]
by the slowly varying property of $L$.  Applying this with $c=\frac12$
and $c=\frac32$, together with 
Proposition~\ref{sum over k prop}, proves Theorem~\ref{main}. \hfill$\Box$

\section{Proof of Theorem 2}

For every positive integer $j$, define
\[
\A_j = \big\{ a\in\N \colon 2^{2^j} < a \le 2^{2^{j+1}},\, 2^j \parallel a \big\},
\]
and define $\A = \bigcup_{j=1}^\infty \A_j$ (a disjoint union). It is
clear that $A(x)\asymp x/\log x$, so that $\A$ has density 0 and
the two assertions in \eqref{eq-sums} hold.
It remains to show that if $\S$ is a primitive subset of $\A$, then 
\eqref{eq-primsum} holds.

Let $\S\subset\A$ be primitive. For each natural number $s$,
define $s^\circ$ to be the largest odd divisor of $s$, 
and define $\S^\circ = \{ s^\circ\colon s\in\S\}$.

\begin{lemma}
If $s_1,s_2\in\S$ are distinct, then $s_1^\circ \nmid s_2^\circ$. In particular, $\S^\circ$ is also primitive, and the map $s\mapsto s^\circ$ is a bijection between $\S$ and $\S^\circ$.
\end{lemma}

\begin{proof}
Suppose, for the sake of contradiction, that $s_1^\circ \mid s_2^\circ$. 
Choose $j_1,j_2\in\N$ so that $s_1\in\A_{j_1}$ and $s_2\in\A_{j_2}$. 
Since $s_1=2^{j_1}s_1^\circ$ and $s_2=2^{j_2}s_2^\circ$, the fact that $s_1\nmid s_2$ (by primitivity of $\S$) forces $j_1 \ge j_2+1$. But then
\[
s_1^\circ = \frac{s_1}{2^{j_1}} > 2^{2^{j_1}-j_1}
\]
and
\[
s_2^\circ = \frac{s_2}{2^{j_2}} \le 2^{2^{j_2+1}-j_2} \le 2^{2^{j_1}-(j_1-1)},
\]
since the expression $2^k-(k-1)$ is an increasing function for $k\ge1$. In particular, $s_2^\circ < 2s_1^\circ$, and so the divisibility relation $s_1^\circ \mid s_2^\circ$ forces $s_1^\circ = s_2^\circ$. But then $s_2\mid s_1$, contradicting the primitivity of~$\S$.

This shows that $s_1^\circ \nmid s_2^\circ$. The symmetric argument shows that $s_2^\circ \nmid s_1^\circ$, and so $\S^\circ$ is indeed primitive. Also, $s_1^\circ \nmid s_2^\circ$ implies that $s_1^\circ \ne s_2^\circ$, which shows that the map $s\mapsto s^\circ$ is a bijection between $\S$ and $\S^\circ$.
\end{proof}

If $s\in\A_j$ then $s^\circ = s/2^j$, and also $2^j \ge (\log s)/(2\log 2)$ by the upper bound on elements of $\A_j$; these relations imply that
\[
s^\circ \log s^\circ = \frac s{2^j}\log\frac s{2^j} \le \frac{2s\log2}{\log s} \log\frac {2s\log2}{\log s} \ll s.
\]
Therefore
\[
\sum_{s\in\S} \frac1s \ll \sum_{s^\circ \in\S^\circ} \frac1{s^\circ \log s^\circ}
\]
(using the injectivity of $s\mapsto s^\circ$). 
However, $\S^\circ$ is primitive, and so the last sum is convergent by \eqref{eq-nlogn}.  
This proves \eqref{eq-primsum}.

\subsection*{Acknowledgements} The first author was supported in part by 
grants from the Natural Sciences and Engineering Research Council of Canada.
The second author was supported in part by NSF grant DMS-0703850.

\bigskip
\noindent
\begin{minipage}[t]{4in}
{Greg Martin\\
Department of Mathematics\\
University of British Columbia\\ 
Room 121, 1984 Mathematics Road \\ 
Vancouver, Canada V6T 1Z2\\
\email{gerg@math.ubc.ca}}
\end{minipage}
\begin{minipage}[t]{2.4in}
{Carl Pomerance\\
Department of Mathematics\\
Dartmouth College\\
Hanover, NH 03755, USA\\
\email{carl.pomerance@dartmouth.edu}}
\end{minipage}

\end{document}